\documentclass[12pt]{amsart}

\usepackage[active]{srcltx}

\makeatletter \@addtoreset{equation}{section}
\@addtoreset{figure}{section}

\makeatother

\usepackage{xypic,amscd}
\usepackage{amsmath}
\usepackage{amssymb}
\usepackage{latexsym}
\newcommand{\mult}{\operatorname{mult}}

\newcommand{\Supp}{\operatorname{Supp}}
\newcommand{\Spec}{\operatorname{Spec}}

\newcommand{\OOO}{{\mathcal O}}
\newcommand{\PP}{\mathbb{P}}

\newcommand{\CC}{\mathbb{C}}
\newcommand{\QQ}{\mathbb{Q}}
\newcommand{\AAA}{\mathbb{A}}

\newtheorem{theorem}[equation]{Theorem}

\newtheorem{lemma}[equation]{Lemma}

\theoremstyle{definition}

\newtheorem{definition}[equation]{Definition}

\newtheorem{remark}[equation]{Remark}

\theoremstyle{remark}

\date{}

\begin{document}

\title[]{Cylinders in del Pezzo surfaces with du Val singularities}
\author[]{Grigory Belousov}
\maketitle

\begin{abstract}
We consider del Pezzo surfaces with du Val singularities. We'll prove that a del Pezzo surface $X$ with du Val singularities has a $-K_X$-polar cylinder if and only if there exist tiger such that the support of this tiger does not contain anti-canonical divisor. Also we classify all del Pezzo surfaces $X$ such that $X$ has not any cylinders.
\end{abstract}

\section{intoduction}

A \emph{log del Pezzo surface} is a projective algebraic surface $X$
with only quotient singularities and ample anti-canonical divisor
$-K_{X}$. In this paper we assume that $X$ has only du Val singularities and we work over complex number field $\CC$. Note that a del Pezzo surface with only du Val singularities is rational.

\begin{definition}
Let $X$ be a proper normal variety. Let $D$ be an effective $\QQ$-divisor on $X$ such that $D\equiv -K_X$ and the pair $(X,D)$ is not log canonical. Such divisor $D$ is called \emph{non-log canonical special tiger} (see {\cite{KeM}}).
\end{definition}

\begin{remark}
In this paper, a non-log canonical special tiger we will call a tiger.
\end{remark}

\begin{definition}[see. {\cite{KPZ}}]
Let $M$ be a $\QQ$-divisor on a projective normal variety $X$. An \emph{$M$-polar cylinder} in $X$ is an open subset $U=X\backslash\Supp(D)$ defined by an effective $\QQ$-divisor $D$ in the $\QQ$-linear equivalence class of $M$ such that $U\cong Z\times\AAA^1$ for some affine variety $Z$.
\end{definition}

In this paper, we consider del Pezzo surfaces with du Val singularities over complex number field $\CC$. Our interest is a connection between existence of a $-K_X$-polar cylinder in the del Pezzo surface and tigers on this surface.

The existence of a $H$-polar cylinder in $X$  is important due to the following fact.

\begin{theorem}[see {\cite{KPZ1}}, Corollary 3.2]
Let $Y$ be a normal algebraic variety over $\CC$ projective over an affine
variety $S$ with $\dim_S Y\geq 1$. Let $H\in  Div(Y)$ be an ample divisor on $Y$, and let $V =
\Spec A(Y,H)$ be the associated affine quasicone over $Y$. Then $V$ admits an effective
$G_a$-action if and only if $Y$ contains an $H$-polar cylinder.
\end{theorem}

There exist a classification of del Pezzo surfaces $X$ such that $X$ has a $-K_X$-polar cylinder (see {\cite{Ch1}}, {\cite{Ch2}}). Also, in the papers {\cite{Ch1}}, {\cite{Ch2}} prove that if a del Pezzo surface $X$ has not $-K_X$-polar cylinder then all tigers contain a support at least one element of $|-K_X|$. Now we prove the inverse statement.

The main result of section 3 is the followings.

\begin{theorem}
\label{glav}
Let $X$ be a del Pezzo surface with du Val singularities. Then $X$ has a $-K_X$-polar cylinder if and only if there exist a tiger such that the support of this tiger does not contain any elements of  $|-K_X|$.
\end{theorem}

The main result of section 4 is the followings.

\begin{theorem}
\label{glav2}
Let $X$ be a del Pezzo surface with du Val singularities. Then
\begin{itemize}
\item
$X$ has not cylinders if $\rho(X)=1$ and $X$ has one of the followings collections of singularities: $4A_2, 2A_1+2A_3, 2D_4$;
\item In the rest cases there exist an ample divisor $H$ such that $X$ has a $H$-polarization.
\end{itemize}
\end{theorem}

The author is grateful to professor I. A. Cheltsov for suggesting
me this problem and for his help.

\section{Preliminary results}

We work over complex number field $\CC$.
We employ the following notation:
\begin{itemize}
\item
$(-n)$-curve is a smooth rational curve with self intersection
number $-n$.
\item
$K_{X}$: the canonical divisor on $X$.
\item
$\rho(X)$: the Picard number of $X$.
\end{itemize}

\begin{theorem}[Riemann--Roch, see, for example, {\cite{Har}}, Theorem 1.6, Ch. 5]
\label{har1}
Let $D$ be a divisor on the surface $X$. Then
$$\chi(D)=\frac{1}{2}D(D-K_X)+\chi(\OOO_X).$$
\end{theorem}

\begin{theorem}[Kawamata--Viehweg Vanishing Theorem, see, for example, {\cite{Mat}}, Theorem 5-2-3]
\label{Van1}
Let $X$ be a non-singular projective variety, $A$ an ample $\QQ$-divisor such that the fractional part $\lceil A\rceil-A$ has the support with only normal crossings. Then
$$H^p(X,K_X+\lceil A\rceil)=0,\quad\text{для любого }p>0.$$
\end{theorem}

Let $X$ be a del Pezzo surface with du Val singularities. Let $d$ be the degree of $X$, i.e. $d=K_X^2$.

\begin{theorem}[see {\cite{Ch1}}, Theorem 1.5]
\label{chel1}
Let $X$ be a del Pezzo surface of degree $d$ with at most du Val singularities.
\begin{itemize}
\item[I.]
The surface $X$ does not admit a $-K_X$-polar cylinder when
\begin{enumerate}
\item $d=1$ and $X$ allows only singular points of types $A_1, A_2, A_3, D_4$ if any;
\item
$d=2$ and $X$ allows only singular points of types $A_1$ if any;
\item
$d=3$ and $X$ allows no singular point.
\end{enumerate}
\item[II.] The surface $X$ has a $-K_X$-polar cylinder if it is not one of the del Pezzo surfaces
listed in I.
\end{itemize}
\end{theorem}

\section{The proof of theorem \ref{glav}}

In the papers {\cite{Ch1}} and {\cite{Ch2}} authors classify del Pezzo surfaces $X$ such that $X$ has a $-K_X$-polar cylinder. Moreover, they prove that if a del Pezzo surface $X$ has not a $-K_X$-polar cylinder then every tiger on $X$ contains an element of $|-K_X|$. So, we need prove that if a del Pezzo surface $X$ has a $-K_X$-polar cylinder then there exist a tiger such that the support of this tiger does not contain any elements of  $|-K_X|$.

\begin{lemma}
\label{Lem1}
Let $X$ be a del Pezzo surface with du Val singularities and let $d$ be the degree of $X$. Assume that $d\geq 7$. Then $X$ has a $-K_X$-polar cylinder and there exist a tiger such that the support of this tiger does not contain any elements of $|-K_X|$.
\end{lemma}

\begin{proof}
By Theorem \ref{chel1}, we see that $X$ has a $-K_X$-polar cylinder. Now, we construct a tiger such that the support of this tiger does not contain any elements of $|-K_X|$.
Consider $|-2K_X|$. By Theorem \ref{har1} and Theorem \ref{Van1}, $\dim|-2K_X|=\frac{-2K_X\cdot(-2K_X-K_X)}{2}=3d$. Let $P$ be an arbitrary smooth point on $X$. Consider a set $\Omega$ of elements $L\in|-2K_X|$ such that $\mult_P L\geq 5$.
 Then $\Omega$ is a linear subsystem of the linear system $|-2K_X|$. Note that $\dim|\Omega|=3d-15\geq 6$ for $d\geq 7$. Hence, $\Omega$ is not empty. Let $N\in\Omega$ be a general element of the linear system $\Omega$.

 Note that $N$ does not contain a support of anti-canonical divisor. Indeed, assume that there exist an element $M_1\in |-K_X|$ such that $\Supp M_1\subseteq\Supp N$. Then $N=M_1+M_2$, where  $M_2\in|-K_X|$. We see that $\dim |-K_X|=\frac{-K_X\cdot(-K_X-K_X)}{2}=d$. Therefore, $\mult_P M_1\leq 3$ and $\mult_P M_2\leq 3$. Hence, we may assume that $\mult_P M_1=2$, $\mult_P M_2=3.$ Let $\tilde{M}_1$ be the linear subsystem of $|-K_X|$ such that $\tilde{M}_1$ consist of elements with multiply two in the point $P$. Let $\tilde{M}_2$ be the linear subsystem of $|-K_X|$ such that $\tilde{M}_2$ consist of elements with multiply three in the point $P$. Then $$\dim|\tilde{M}_1+\tilde{M}_2|=\dim|\tilde{M}_1|+\dim|\tilde{M}_2|=(d-3)+(d-6)=2d-9.$$ Note that $3d-15>2d-9$ for $d\geq 7$. Hence, a general element $N$ of the linear system $\Omega$ does not contain a support of anti-canonical divisor. Then $\frac{1}{2}N$ is a tiger such that the support of this tiger does not contain any elements of $|-K_X|$.
\end{proof}

\begin{lemma}
\label{Lem2}
Let $X$ be a del Pezzo surface with du Val singularities and let $d$ be the degree of $X$. Assume that $d=4,6$. Then $X$ has a $-K_X$-polar cylinder and there exist a tiger such that the support of this tiger does not contain any elements of $|-K_X|$.
\end{lemma}

\begin{proof}
By Theorem \ref{chel1}, we see that $X$ has a $-K_X$-polar cylinder. Now, we construct a tiger such that the support of this tiger does not contain any elements of $|-K_X|$.
Let $f:\bar{X}\rightarrow X$ be the minimal resolution. Let $E$ be a $(-1)$-curve on $\bar{X}$ and $E'=f(E)$. Put $-3K_{\bar{X}}\backsim 2E+F$. Then $-3K_{\bar{X}}\cdot E=2E^2+F\cdot E$. Since $K_{\bar{X}}\cdot E=-1$ and $E^2=-1$, we see that $F\cdot E=5$. We have $-3K_{\bar{X}}^2=2E\cdot K_{\bar{X}}+F\cdot K_{\bar{X}}$. Since $K_{\bar{X}}\cdot E=-1$ and $K_{\bar{X}}^2=d$, we see that $F\cdot K_{\bar{X}}=-(3d-2)$. We obtain $-3K_{\bar{X}}\cdot F= 2E\cdot F+F^2$. Since $F\cdot E=5$ and $F\cdot K_{\bar{X}}=-(3d-2)$ we see that $F^2=9d-16$.
Hence, by Theorem \ref{har1} and Theorem \ref{Van1}, $\dim |F|=6d-9$. Let $P'$ be a general smooth point on $E'$ and $P'=f(P)$. Consider a set $\Omega$ of elements $L\in|F|$ such that $\mult_P L\geq 5$.
Note that $\dim|\Omega|=6d-9-15=6d-24\geq 0$ for $d\geq 4$, i.e. $\Omega$ is non-empty. We see that $\Omega$ contains an element $N$ such that $N+E$ does not contain a support of anti-canonical divisor. Indeed, assume that for all $N\in\Omega$ there exist  $M_1\in |-K_{\bar{X}}|$ such that $\Supp M_1\subseteq\Supp (N+E)$. Then $N+2E=M_1+M_2$, where
$M_2\in|-2K_{\bar{X}}|$. We have the followings three cases.

Case 1. $M_1=2E+F_1$, $M_2$ does not contain the curve $E$. Hence, $F_1\cdot E=3$, $F_1\cdot K_{\bar{X}}=-(d-2)$, $F_1^2=d-8\leq -2$, a contradiction.

Case 2. $M_1=E+F_1$, $M_2=E+F_2$. Then $F_1\cdot E=2$, $F_1\cdot K_{\bar{X}}=-(d-1)$, $F_1^2=d-3$, $F_2\cdot E=3$, $F_2\cdot K_{\bar{X}}=-(2d-1)$, $F_2^2=4d-5$. Hence, $\dim|F_1|=d-2$, $\dim |F_2|=3d-3$.
Note that the multiplicities $F_1$ and $F_2$ in the point $P$ are equaled 2 and 3 correspondingly. Let $\tilde{F}_1$ be the linear subsystem of $|F_1|$ such that the multiplicity of elements of $\tilde{F}_1$ is equaled two in the point $P$, let $\tilde{F}_2$ be the linear subsystem of $|F_2|$ such that the multiplicity of elements of $\tilde{F}_2$ is equaled three in the point $P$. Then $\dim|\tilde{F}_1|=d-5.$ Hence, $d=6$. Note that $$\dim|\tilde{F}_1+\tilde{F}_2|=\dim|\tilde{F}_1|+\dim|\tilde{F}_2|=(d-5)+(3d-9)=4d-14=10.$$ On the other hand, $\dim|\Omega|=6d-24=12>10$. Therefore, a general element $N\in\Omega$ does not contain $\Supp(-K_{\bar{X}})\backslash\Supp(E)$.

Case 3. $M_2=2E+F_2$, $M_1$ does not contain the curve $E$. Then $F_2\cdot E=4$, $F_2\cdot K_{\bar{X}}=-(2d-2)$, $F_2^2=4d-12$. Hence, $\dim|F_2|=3d-7$, $\dim|M_1|=d$. Note that the multiplicities $M_1$ and $F_2$ in the poin $P$ are equal to 1 and 4 correspondingly. Let $\tilde{M}_1$ be the set of elements of the linear system $|-K_{\bar{X}}|$ that pass through the point $P$, let  $\tilde{F}_2$ be the set of elements of the linear system $|F_2|$ that have multiplicity four in the point $P$. Note that $\tilde{F}_1$ and $\tilde{M}_2$ are the linear system. Then $\dim|\tilde{F}_2|=3d-17.$ Hence, $d=6$. Note that $$\dim|\tilde{M}_1+\tilde{F}_2|=\dim|\tilde{M}_1|+\dim|\tilde{F}_2|=(d-1)+(3d-17)=4d-18=6.$$ On the other hand, $\dim|\Omega|=6d-24=12>6$. Therefore, a general element $N\in\Omega$ does not contain any elements of $|-K_{\bar{X}}-E|$.

So, a general element $N\in\Omega$ does not contain any elements of $|-K_{\bar{X}}-E|$. Denote this element by $N$. Note that $\mult_P(2E+N)\geq 7$. Then $\frac{1}{3}f(N)+\frac{2}{3}E'$ is a tiger such that the support of this tiger does not contain any elements of $|-K_X|$.
\end{proof}

\begin{lemma}
\label{Lem3}
Let $X$ be a del Pezzo surface with du Val singularities and let $d$ be the degree of $X$. Assume that $d=5$. Then $X$ has a $-K_X$-polar cylinder and there exist a tiger such that the support of this tiger does not contain any elements of $|-K_X|$.
\end{lemma}

\begin{proof}
By Theorem \ref{chel1}, we see that $X$ has a $-K_X$-polar cylinder. Now, we construct a tiger such that the support of this tiger does not contain any elements of $|-K_X|$.
Consider $|-4K_X|$. By Theorem \ref{har1} and Theorem \ref{Van1}, we see that $\dim|-4K_X|=50$. Let $P$ be an arbitrary smooth point on $X$. Consider a set $\Omega$ of elements $L\in|-4K_X|$ such that $\mult_P L\geq 9$.
Then $\Omega$ is the linear subsystem of the linear system of $|-4K_X|$. Note that $\dim|\Omega|=50-45=5$. Hence, $\Omega$ is non-empty. Let $N\in\Omega$ be a general element of the linear system $\Omega$. We see that $N$ does not contain a support of anti-canonical divisor. Indeed, assume that there exist an element $M_1\in |-K_X|$ such that $\Supp M_1\subseteq\Supp N$. Then $N=M_1+M_2$, where
$M_2\in|-3K_X|$. Note that $\dim |-K_X|=5$, $\dim |-3K_X|=30$. Put $d_1=\mult_P M_1$ and $d_2=\mult_P M_2$. Since $$\dim |-K_X|-\frac{d_1\cdot (d_1+1)}{2}=5-\frac{d_1\cdot (d_1+1)}{2}\geq 0$$ and $$\dim |-3K_X|-\frac{d_2\cdot (d_2+1)}{2}=30-\frac{d_2\cdot (d_2+1)}{2}\geq 0,$$ we see that $\mult_P M_1\leq 2$ and $\mult_P M_1\leq 7$. Hence, $\mult_P M_1=2$, $\mult_P M_2=7$. Let $\tilde{M}_1$ be the set of elements of the linear system $|-K_X|$ that have multiply $2$ in the point $P$, let $\tilde{M}_2$ be the set of elements of the linear system $|-3K_X|$ that have multiply $7$ in the point $P$. Note that $\tilde{M}_1$ and $\tilde{M}_2$ are the linear system. Then $\dim|\tilde{M}_1|=5-3=2$ $\dim|\tilde{M}_2|=30-28=2$. Hence, $$\dim|\tilde{M}_1+\tilde{M}_2|=4<5=\dim|\Omega|.$$ So, a general element $N$ of $\Omega$ does not contain the support of anti-canonical divisor. Then $\frac{1}{4}N$ is a tiger such that the support of this tiger does not contain any elements of $|-K_X|$.
\end{proof}

\begin{lemma}
\label{Lem4}
Let $X$ be a del Pezzo surface with du Val singularities and let $d$ be the degree of $X$. Assume that $d\geq3$ and there exist a singular point of type $A_1$. Then $X$ has a $-K_X$-polar cylinder and there exist a tiger such that the support of this tiger does not contain any elements of $|-K_X|$.
\end{lemma}

\begin{proof}
Let $X$ be a del Pezzo surface with du Val singularities, and let $P$ be a singular point of type $A_1$. By Lemmas \ref{Lem1}, \ref{Lem2} and \ref{Lem3} we may assume that $d=3$. Let $f:\bar{X}\rightarrow X$ be the minimal resolution of singularities of $X$, and let $D=\sum_{i=1}^n D_i$ be the exceptional divisor of $f$, where $D_i$ is a $(-2)$-curve. We may assume that $P=f(D_1)$. By Theorem \ref{chel1}, we see that $X$ has a $-K_X$-polar cylinder. Now, we construct a tiger such that the support of this tiger does not contain any elements of $|-K_X|$. Consider $-4K_{\bar{X}}$. Put $-4K_{\bar{X}}\backsim 3D_1+F$. Then $F\cdot D_1=6$, $F\cdot K_{\bar{X}}=-12$, $F^2=30$. Hence, $\dim|F|=21$. Let $Q$ be a point on $D_1$. Note that there exist an element $N\in |F|$ such that $\mult_Q N=6$. Now, we prove that $N+D_1$ does not contain the support of anti-canonical divisor. Indeed, assume that for all $N\in\Omega$ there exist an element $M_1\in |-K_{\bar{X}}|$ such that $\Supp M_1\subseteq\Supp (N+D_1)$. Then $N+3D_1=M_1+M_2$, where $M_2\in|-3K_{\bar{X}}|$. So, we have the following four cases.

Case 1. $M_2=3D_1+F_2$, $M_1$ does not contain the curve $D_1$. Then $F_2\cdot D_1=6$, $F_2\cdot K_X=-9$, $F_2^2=9$. Hence, $\dim|F_2|=9$. Therefore, $\mult_Q F_2\leq 3$. Since $M_1$ does not meet $D_1$, we have a contradiction.

Case 2. $M_1=D_1+F_1$, $M_2=2D_1+F_2$. Then $F_1\cdot D_1=2$, $F_1\cdot K_X=-d$, $F_1^2=1$. Therefore, $\dim|F_1|= 2$. Hence, $\mult_Q F_2\leq 1$, a contradiction.

Case 3. $M_1=2D_1+F_1$, $M_2=D_1+F_2$. Then $F_1\cdot D_1=4$, $F_1\cdot K_X=-3$, $F_1^2=-5$, a contradiction.

Case 4. $M_1=3D_1+F_2$, $M_2$ does not contain the curve $D_1$. Then $F_1\cdot D_1=6$, $F_1\cdot K_X=-3$, $F_1^2=-15$, a contradiction.

So, $\Supp(N+D_1)$ does not contain the support of anti-canonical divisor. Note that $\mult_Q(3D_1+N)=9$.  Then $\frac{1}{4} f(N)$ is a tiger such that the support of this tiger does not contain any elements of $|-K_X|$.
\end{proof}

\begin{lemma}
\label{LemA1A2}
Let $X$ be a del Pezzo surface with du Val singularities and let $d$ be the degree of $X$. Assume that $d\geq 2$ and there exist a singular point of type $A_2$ or $A_3$. Then $X$ has a $-K_X$-polar cylinder and there exist a tiger such that the support of this tiger does not contain any elements of $|-K_X|$.
\end{lemma}

\begin{proof}
As above, we may assume that $d=2$ or $d=3$. By Theorem \ref{chel1}, we see that $X$ contains $-K_X$-polar cylinder. Let $f:\bar{X}\rightarrow X$ be the minimal resolution of singularities of $X$, and let $D=\sum_{i=1}^n D_i$ be the exceptional divisor of $f$, where $D_i$ is a $(-2)$-curve. Consider two cases.

Case 1. There exist a point $P\in X$ such that $P$ of type $A_2$. We may assume that $D_1$ and $D_2$ correspond to $P$. So, $D_1\cdot D_2=1$. Let $Q$ be the point of intersection of $D_1$ and $D_2$. Consider $-2K_{\bar{X}}$. Put $-2K_{\bar{X}}\backsim 2D_1+2D_2+F$. Then $F\cdot D_1=F\cdot D_2=2$, $F\cdot K_{\bar{X}}=-2d$, $F^2=4d-8$. Hence, $\dim|F|=3d-4$. Consider the set $\Omega$ of elements $L\in|F|$ such that $Q\in L$. Then $\dim\Omega =3d-4-1=3d-5$. Put $-K_{\bar{X}}\backsim D_1+D_2+\tilde{F}$. Then $\tilde{F}\cdot D_1=\tilde{F}\cdot D_2=1$, $\tilde{F}\cdot K_{\bar{X}}=-d$, $\tilde{F}^2=d-2$. Hence, $|\tilde{F}|=d-1$. Consider the set $\tilde{\Omega}$ of elements $L\in|\tilde{F}|$ such that $Q\in L$. Then $\dim\tilde{\Omega}=d-2$. Note that $\dim\Omega=3d-5>\dim\tilde{\Omega}=d-2$. So, the there exist an element $N$ of $\Omega$ such that $f(N)$ does not contain the support of anti-canonical divisor. Note that $\mult_Q(2D_1+2D_2+N)\geq 5$. Then $\frac{1}{2} f(N)$ is a tiger such that the support of this tiger does not contain any elements of $|-K_X|$.

Case 2. There exist a point $P\in X$ such that $P$ of type $A_3$. We may assume that $D_1$, $D_2$ and $D_3$ correspond to $P$. So, $D_1\cdot D_2=D_2\cdot D_3=1$. Let $Q$ be the point of intersection of $D_1$ and $D_2$. Consider $-2K_{\bar{X}}$. Put $-2K_{\bar{X}}\backsim 2D_1+2D_2+D_3+F$. Then $F\cdot D_1=2$, $F\cdot D_2=1$, $F\cdot D_3=0$ $F\cdot K_{\bar{X}}=-2d$, $F^2=4d-6$. Hence, $\dim|F|=3d-3$.
So, there exist an element $N$ of $|F|$ such that $Q\in N$. Note that the support $N+2D_1+2D_2+D_3$ does not contain any elements $|-K_{\bar{X}}|$. So, $f(N)$ does not contain the support of anti-canonical divisor. Note that $\mult_Q(2D_1+2D_2+D_3+N)\geq 5$. Then $\frac{1}{2} f(N)$ is a tiger such that the support of this tiger does not contain any elements of $|-K_X|$.
\end{proof}

\begin{lemma}
\label{LemD4}
Let $X$ be a del Pezzo surface with du Val singularities and let $d$ be the degree of $X$. Assume that $d\geq 2$ and there exist a singular point of type $D_4$. Then $X$ has a $-K_X$-polar cylinder and there exist a tiger such that the support of this tiger does not contain any elements of $|-K_X|$.
\end{lemma}

\begin{proof}
Let $X$ be a del Pezzo surface with du Val singularities, and let $P$ be a singular point of type $D_4$. By Theorem \ref{chel1}, we see that $X$ has a $-K_X$-polar cylinder. Let $f:\bar{X}\rightarrow X$ be the minimal resolution of singularities of $X$, and let $D=\sum_{i=1}^n D_i$ be the exceptional divisor of $f$, where $D_i$ is a $(-2)$-curve. We may assume that $D_1$, $D_2$, $D_3$ and $D_4$ correspond to $P$. Moreover, $D_1$ is the central component. Put $-3K_{\bar{X}}\backsim 4D_1+3D_2+2D_3+2D_4+F$. Then $F\cdot D_1=1$, $F\cdot D_2=2$, $F\cdot D_3=F\cdot D_4=0$ $F\cdot K_{\bar{X}}=-3d$, $F^2=9d-10>0$ for $d\geq 2$. Note that $4D_1+3D_2+2D_3+2D_4+F$ does not admit representation as $M_1+M_2$, where $M_1\in|-K_{\bar{X}}|$ and $M_2\in|-2K_{\bar{X}}|$. Let $N$ be an element of $|F|$. Note that the support $N+4D_1+3D_2+2D_3+2D_4$ does not contain any elements $|-K_{\bar{X}}|$. So, $f(N)$ does not contain the support of anti-canonical divisor. Note that $\mult_Q(4D_1+3D_2+2D_3+2D_4+N)\geq 7$, where $Q$ is the intersection of $D_1$ and $D_2$. Then $\frac{1}{3} f(N)$ is a tiger such that the support of this tiger does not contain any elements of $|-K_X|$.
\end{proof}

\begin{lemma}
\label{LemA4-8}
Let $X$ be a del Pezzo surface with du Val singularities and let $d$ be the degree of $X$. Assume that there exist a singular point of type $A_k$, where $k=4,5,6,7,8$. Then $X$ has a $-K_X$-polar cylinder and there exist a tiger such that the support of this tiger does not contain any elements of $|-K_X|$.
\end{lemma}

\begin{proof}
Let $X$ be a del Pezzo surface with du Val singularities, and let $P$ be a singular point of type $A_k$. By Theorem \ref{chel1}, we see that $X$ has a $-K_X$-polar cylinder. Let $f:\bar{X}\rightarrow X$ be the minimal resolution of singularities of $X$, and let $D=\sum_{i=1}^n D_i$ be the exceptional divisor of $f$, where $D_i$ is a $(-2)$-curve. We may assume that $D_1, D_2,\ldots, D_k$ correspond to $P$. Moreover, $D_i\cdot D_{i+1}=1$ for $i=1,2,\ldots,k-1$. Consider the following cases.

Case 1. $k=4$. Put $-2K_{\bar{X}}\backsim D_1+2D_2+2D_3+D_4+F$. Let $Q$ be the intersection of $D_2$ and $D_3$. We obtain $F\cdot D_1=F\cdot D_4=0$, $F\cdot D_2=F\cdot D_3=1$, $F\cdot K_{\bar{X}}=-2d$, $F^2=4d-4$. Then $\dim|F|=3d-2$. So, there exist an element $N\in|F|$ such that $N$ pass through $Q$. Note that $D_1+2D_2+2D_3+D_4+N$ does not admit representation as $M_1+M_2$, where $M_1, M_2\in|-K_{\bar{X}}|$. Hence, $f(N)$ does not contain the support of anti-canonical divisor. Note that $\mult_Q(D_1+2D_2+2D_3+D_4+N)\geq 5$. Then $\frac{1}{2} f(N)$ is a tiger such that the support of this tiger does not contain any elements of $|-K_X|$.

Case 2. $k=5$. Put $-3K_{\bar{X}}\backsim D_1+2D_2+3D_3+3D_4+2D_5+F$. Let $Q$ be the intersection of $D_3$ and $D_4$. We obtain $F\cdot D_1=F\cdot D_2=0$, $F\cdot D_3=F\cdot D_4=F\cdot D_5=1$, $F\cdot K_{\bar{X}}=-3d$, $F^2=9d-8$. Then $\dim|F|=6d-4$. So, there exist an element $N\in|F|$ such that $N$ pass through $Q$. Note that $D_1+2D_2+3D_3+3D_4+2D_5+N$ does not admit representation as $M_1+M_2$, where $M_1\in|-K_{\bar{X}}|$ and $M_2\in|-2K_{\bar{X}}|$. Hence, $f(N)$ does not contain the support of anti-canonical divisor. Note that $\mult_Q(D_1+2D_2+3D_3+3D_4+2D_5+N)\geq 7$.  Then $\frac{1}{3} f(N)$ is a tiger such that the support of this tiger does not contain any elements of $|-K_X|$.

Case 3. $k=6$. Put $-3K_{\bar{X}}\backsim D_1+2D_2+3D_3+3D_4+2D_5+D_6+F$. Let $Q$ be the intersection of $D_3$ and $D_4$. We obtain $$F\cdot D_1=F\cdot D_2=F\cdot D_5=F\cdot D_6=0,$$ $$F\cdot D_3=F\cdot D_4=1, F\cdot K_{\bar{X}}=-3d, F^2=9d-6.$$ Then $\dim|F|=6d-3$. So, there exist an element $N\in|F|$ such that $N$ pass through $Q$. Note that $D_1+2D_2+3D_3+3D_4+2D_5+N$ does not admit representation as $M_1+M_2$, where $M_1\in|-K_{\bar{X}}|$ and $M_2\in|-2K_{\bar{X}}|$. Hence, $f(N)$ does not contain the support of anti-canonical divisor. Note that $\mult_Q(D_1+2D_2+3D_3+3D_4+2D_5+D_6+N)\geq 7$.  Then $\frac{1}{3} f(N)$ is a tiger such that the support of this tiger does not contain any elements of $|-K_X|$.

Case 4. $k=7$. Put $$-4K_{\bar{X}}\backsim D_1+2D_2+3D_3+4D_4+4D_5+3D_6+2D_7+F.$$ Let $Q$ be the intersection of $D_4$ and $D_5$. We obtain $$F\cdot D_1=F\cdot D_2=F\cdot D_3=F\cdot D_6=0,$$ $$F\cdot D_4=F\cdot D_5=F\cdot D_7=1, F\cdot K_{\bar{X}}=-4d, F^2=16d-10.$$ Then $\dim|F|=10d-5$. So, there exist an element $N\in|F|$ such that $N$ pass through $Q$. Note that $D_1+2D_2+3D_3+4D_4+4D_5+3D_6+2D_7+N$ does not admit representation as $M_1+M_2$, where $M_1\in|-K_{\bar{X}}|$ and $M_2\in|-3K_{\bar{X}}|$. Hence, $f(N)$ does not contain the support of anti-canonical divisor. Note that $$\mult_Q(D_1+2D_2+3D_3+4D_4+4D_5+3D_6+2D_7+N)\geq 9.$$ Then $\frac{1}{4} f(N)$ is a tiger such that the support of this tiger does not contain any elements of $|-K_X|$.

Case 5. $k=8$. Put $$-4K_{\bar{X}}\backsim D_1+2D_2+3D_3+4D_4+4D_5+3D_6+2D_7+D_8+F.$$ Let $Q$ be the intersection of $D_4$ and $D_5$. We obtain $$F\cdot D_1=F\cdot D_2=F\cdot D_3=F\cdot D_6=F\cdot D_7=F\cdot D_8=0,$$ $$F\cdot D_4=F\cdot D_5=1, F\cdot K_{\bar{X}}=-4d, F^2=16d-8.$$ Then $\dim|F|=10d-4$. So, there exist an element $N\in|F|$ such that $N$ pass through $Q$. Note that $D_1+2D_2+3D_3+4D_4+4D_5+3D_6+2D_7+D_8+N$ does not admit representation as $M_1+M_2$, where $M_1\in|-K_{\bar{X}}|$ and $M_2\in|-3K_{\bar{X}}|$. Hence, $f(N)$ does not contain the support of anti-canonical divisor. Note that $$\mult_Q(D_1+2D_2+3D_3+4D_4+4D_5+3D_6+2D_7+D_8+N)\geq 9.$$ Then $\frac{1}{4} f(N)$ is a tiger such that the support of this tiger does not contain any elements of $|-K_X|$.
\end{proof}

\begin{lemma}
\label{LemD5-8}
Let $X$ be a del Pezzo surface with du Val singularities and let $d$ be the degree of $X$. Assume that there exist a singular point of type $D_k$, where $k=5,6,7,8$. Then $X$ has a $-K_X$-polar cylinder and there exist a tiger such that the support of this tiger does not contain any elements of $|-K_X|$.
\end{lemma}

\begin{proof}
Let $X$ be a del Pezzo surface with du Val singularities, and let $P$ be a singular point of type $D_k$. By Theorem \ref{chel1}, we see that $X$ has a $-K_X$-polar cylinder. Let $f:\bar{X}\rightarrow X$ be the minimal resolution of singularities of $X$, and let $D=\sum_{i=1}^n D_i$ be the exceptional divisor of $f$, where $D_i$ is a $(-2)$-curve. We may assume that $D_1, D_2,\ldots, D_k$ correspond to $P$. Moreover, $D_3$ is the central component, $D_1$, $D_2$ meet only $D_3$, and $D_i\cdot D_{i+1}=1$ for $i=3,4,\ldots,k-1$. Consider the following cases.

Case 1. $k=5$. Put $-2K_{\bar{X}}\backsim 2D_1+2D_2+3D_3+2D_4+D_5+F$. Then $F\cdot D_1=F\cdot D_2=1$, $F\cdot D_3=F\cdot D_4=F\cdot D_5=0$, $F\cdot K_{\bar{X}}=-2d$, $F^2=4d-4$. Then $\dim|F|=3d-2$. So, there exist an element $N\in|F|$. Note that $2D_1+2D_2+3D_3+2D_4+D_5+N$ does not admit representation as $M_1+M_2$, where $M_1, M_2\in|-K_{\bar{X}}|$. Hence, $f(N)$ does not contain the support of anti-canonical divisor. Note that $\mult_Q(2D_1+2D_2+3D_3+2D_4+D_5+N)\geq 5$, where $Q$ is the intersection of $D_3$ and $D_4$. Then $\frac{1}{2} f(N)$ is a tiger such that the support of this tiger does not contain any elements of $|-K_X|$.

Case 2. $k=6$. Put $-2K_{\bar{X}}\backsim 2D_1+2D_2+4D_3+3D_4+2D_5+D_6+F$. Then $F\cdot D_3=1$, $F\cdot D_i=0$ for $i\neq 3$, $F\cdot K_{\bar{X}}=-2d$, $F^2=4d-4$. Then $\dim|F|=3d-2$. So, there exist an element $N\in|F|$. Note that $2D_1+2D_2+4D_3+3D_4+2D_5+D_6+N$ does not admit representation as $M_1+M_2$, where $M_1, M_2\in|-K_{\bar{X}}|$. Hence, $f(N)$ does not contain the support of anti-canonical divisor. Note that $\mult_Q(2D_1+2D_2+4D_3+3D_4+2D_5+D_6+N)\geq 7$, where $Q$ is the intersection of $D_3$ and $D_4$. Then $\frac{1}{2} f(N)$ is a tiger such that the support of this tiger does not contain any elements of $|-K_X|$.

Case 3. $k=7$. Put $$-3K_{\bar{X}}\backsim 3D_1+3D_2+6D_3+5D_4+4D_5+3D_6+2D_7+F.$$ Then $F\cdot D_3=F\cdot D_7=1$, $F\cdot D_i=0$ for $i\neq 3,7$, $F\cdot K_{\bar{X}}=-3d$, $F^2=9d-8$. Then $\dim|F|=6d-4$. So, there exist an element $N\in|F|$. Note that $3D_1+3D_2+6D_3+5D_4+4D_5+3D_6+2D_7+N$ does not admit representation as $M_1+M_2$, where $M_1\in|-K_{\bar{X}}|$ and $M_2\in|-2K_{\bar{X}}|$. Hence, $f(N)$ does not contain the support of anti-canonical divisor. Note that $\mult_Q(3D_1+3D_2+6D_3+5D_4+4D_5+3D_6+2D_7+N)\geq 11$, where $Q$ is the intersection of $D_3$ and $D_4$. Then $\frac{1}{3} f(N)$ is a tiger such that the support of this tiger does not contain any elements of $|-K_X|$.

Case 4. $k=8$. Put $$-3K_{\bar{X}}\backsim 3D_1+3D_2+6D_3+5D_4+4D_5+3D_6+2D_7+D_8+F.$$ Then $F\cdot D_3=1$, $F\cdot D_i=0$ for $i\neq 3$, $F\cdot K_{\bar{X}}=-3d$, $F^2=9d-6$. Then $\dim|F|=6d-3$. So, there exist an element $N\in|F|$. Note that $3D_1+3D_2+6D_3+5D_4+4D_5+3D_6+2D_7+D_8+N$ does not admit representation as $M_1+M_2$, where $M_1\in|-K_{\bar{X}}|$ and $M_2\in|-2K_{\bar{X}}|$. Hence, $f(N)$ does not contain the support of anti-canonical divisor. Note that $\mult_Q(3D_1+3D_2+6D_3+5D_4+4D_5+3D_6+2D_7+D_8+N)\geq 11$, where $Q$ is the intersection of $D_3$ and $D_4$. Then $\frac{1}{3} f(N)$ is a tiger such that the support of this tiger does not contain any elements of $|-K_X|$.
\end{proof}

\begin{lemma}
\label{LemE}
Let $X$ be a del Pezzo surface with du Val singularities and let $d$ be the degree of $X$. Assume that there exist a singular point of type $E_k$, where $k=6,7,8$. Then $X$ has a $-K_X$-polar cylinder and there exist a tiger such that the support of this tiger does not contain any elements of $|-K_X|$.
\end{lemma}

\begin{proof}
Let $X$ be a del Pezzo surface with du Val singularities, and let $P$ be a singular point of type $D_k$. By Theorem \ref{chel1}, we see that $X$ has a $-K_X$-polar cylinder. Let $f:\bar{X}\rightarrow X$ be the minimal resolution of singularities of $X$, and let $D=\sum_{i=1}^n D_i$ be the exceptional divisor of $f$, where $D_i$ is a $(-2)$-curve. We may assume that $D_1, D_2,\ldots, D_k$ correspond to $P$. Moreover, $D_4$ is the central component, $D_1$ meets only $D_4$, $D_3$ meets $D_2$ and $D_4$, $D_2$ meets only $D_3$, and $D_i\cdot D_{i+1}=1$ for $i=3,4,\ldots,k-1$. Consider the following cases.

Case 1. $k=6$. Put $-2K_{\bar{X}}\backsim 2D_1+D_2+2D_3+3D_4+2D_5+D_6+F$. Then $F\cdot D_1=1$, $F\cdot D_i=0$ for $i\geq 2$, $F\cdot K_{\bar{X}}=-2d$, $F^2=4d-2$. Then $\dim|F|=3d-1$. So, there exist an element $N\in|F|$. Note that $2D_1+D_2+2D_3+3D_4+2D_5+D_6+N$ does not admit representation as $M_1+M_2$, where $M_1, M_2\in|-K_{\bar{X}}|$. Hence, $f(N)$ does not contain the support of anti-canonical divisor. Note that $\mult_Q(2D_1+D_2+2D_3+3D_4+2D_5+D_6+N)\geq 5$, where $Q$ is the intersection of $D_4$ and $D_5$.  Then $\frac{1}{2} f(N)$ is a tiger such that the support of this tiger does not contain any elements of $|-K_X|$.

Case 2. $k=7$. Put $$-2K_{\bar{X}}\backsim 2D_1+2D_2+3D_3+4D_4+3D_5+2D_6+D_7+F.$$ Then $F\cdot D_2=1$, $F\cdot D_i=0$ for $i\neq 2$, $F\cdot K_{\bar{X}}=-2d$, $F^2=4d-2$. Then $\dim|F|=3d-1$. So, there exist an element $N\in|F|$. Note that $2D_1+2D_2+3D_3+4D_4+3D_5+2D_6+D_7+N$ does not admit representation as $M_1+M_2$, where $M_1, M_2\in|-K_{\bar{X}}|$. Hence, $f(N)$ does not contain the support of anti-canonical divisor. Note that $\mult_Q(2D_1+2D_2+3D_3+4D_4+3D_5+2D_6+D_7+N)\geq 7$, where $Q$ is the intersection of $D_4$ and $D_5$. Then $\frac{1}{2} f(N)$ is a tiger such that the support of this tiger does not contain any elements of $|-K_X|$.

Case 3. $k=8$. Put $$-2K_{\bar{X}}\backsim 3D_1+2D_2+4D_3+6D_4+5D_5+4D_6+3D_7+2D_8+F.$$ Then $F\cdot D_8=1$, $F\cdot D_i=0$ for $i\neq 8$, $F\cdot K_{\bar{X}}=-2d$, $F^2=4d-2$. Then $\dim|F|=3d-1$. So, there exist an element $N\in|F|$. Note that $3D_1+2D_2+4D_3+6D_4+5D_5+4D_6+3D_7+2D_8+N$ does not admit representation as $M_1+M_2$, where $M_1, M_2\in|-K_{\bar{X}}|$. Hence, $f(N)$ does not contain the support of anti-canonical divisor. Note that $\mult_Q(2D_1+2D_2+3D_3+4D_4+3D_5+2D_6+D_7+N)\geq 7$, where $Q$ is the intersection of $D_4$ and $D_5$. Then $\frac{1}{2} f(N)$ is a tiger such that the support of this tiger does not contain any elements of $|-K_X|$.
\end{proof}

So, Theorem \ref{glav} follows from lemmas \ref{Lem1}, \ref{Lem2}, \ref{Lem3}, \ref{Lem4}, \ref{LemA1A2}, \ref{LemD4}, \ref{LemA4-8}, \ref{LemD5-8}, \ref{LemE}.

\section{The proof of theorem \ref{glav2}}

Assume that $\rho(X)=1$. Then $X$ has a $H$-polar cylinder if and only if $X$ has a $-K_X$-polar cylinder, where $H$ is an arbitrary ample divisor. On the other hand, there exist a classification of del Pezzo surfaces $X$ such that $X$ has a $-K_X$-polar cylinder (see {\cite{Ch1}}). By a classification of a del Pezzo surface $X$ has not cylinders if $X$ has one of the following collections of singularities: $4A_2, 2A_1+2A_3, 2D_4$. So, we may assume that $\rho(X)>1$.

Let $f:\bar{X}\rightarrow X$ be the minimal resolution of singularities of $X$, and let $D=\sum_{i=1}^n D_i$ be the exceptional divisor of $f$, where $D_i$ is a $(-2)$-curve.

\begin{lemma}
\label{Lemm1}
Assume that there exist a $\PP^1$-fibration $g\colon\bar{X}\rightarrow\PP^1$ such that at most one irreducible component of the exceptional divisor $D$ not contained in any
fiber of $g$. Moreover, this component is an $1$-section. Then there exist an ample divisor $H$ such that $X$ has a $H$-polar cylinder.
\end{lemma}

\begin{proof}
Let $F$ be a unique exception curve not contained in any
fiber of $g$ (if there exist no such component, then $F$ is an arbitrary $1$-section). Put $$-K_{\bar{X}}\sim_{\QQ} 2F+\sum a_i E_i.$$ Note that all $E_i$ are contained in fibers of $g$. Consider an ample divisor $H=-K_{\bar{X}}+mC$, where $C$ is a fiber of $g$, $m$ is a sufficiently large number. Then there exist a divisor $\hat{H}\sim_{\QQ} H$ such that $$\hat{H}=2F+\sum b_i \hat{E}_i,$$ where $b_i>0$ and the set of $\hat{E}_i$ contains all irreducible curves in singular fibers of $g$. Then $$\bar{X}\setminus\Supp(\hat{H})\cong\AAA^1\times(\PP^1\setminus\{p_1,\ldots,p_k\}),$$ where $p_1,\ldots,p_k$ correspond to singular fibers of $g$. So, $\bar{X}$ has a $H$-polarization. Hence, $X$ has a $f(\hat{H})$-polarization.
\end{proof}

Run MMP for $X$. We obtain $$X=X_1\rightarrow X_2\rightarrow\cdots\rightarrow X_n.$$ Assume that $X_n=\PP^1$. Consider the composition of the minimal resolution and MMP. We have a $\PP^1$-fibration $g\colon\bar{X}\rightarrow\PP^1$. Note that all exception curves of $f$ are contained in fibers of $g$. Hence, by Lemma \ref{Lemm1}, we see that there exist an ample divisor $H$ such that $X$ has a $H$-polar cylinder.

So, we may assume that $X_n$ is a del Pezzo surface with $\rho(X_n)=1$ and du Val singularities.

\begin{lemma}
\label{Lemm2}
Assume that $X_n$ has a $-K_{X_n}$-polar cylinder. Then there exist an ample divisor $H$ such that $X$ has a $H$-polar cylinder.
\end{lemma}

\begin{proof}
Put $h\colon X\rightarrow X_n$. Assume that $h$ contracts extremal rays in points $p_1,p_2,\ldots,p_m$. Let $M$ be an anti-canonical divisor such that $X_n\setminus\Supp(M)\cong Z\times\AAA^1$. Let $\phi\colon X_n\setminus\Supp(M)\rightarrow Z$ be the projection on first factor. Let $C_1,C_2,\ldots,C_k$ be the fibers of $\phi$ such that $C_1,C_2,\ldots,C_k$ contain $p_1,p_2,\ldots,p_m$, and let $\bar{C}_1,\bar{C}_2,\ldots,\bar{C}_k$ be the closure of $C_1,C_2,\ldots,C_k$ on $X_n$. Since $\rho(X_n)=1$, we see that $C_i\sim_{\QQ} -a_i K_{X_n}$. Consider the divisor $$L=M+m_1\bar{C}_1+m_2\bar{C}_2+\cdots+m_k\bar{C}_k,$$ where $m_1,m_2,\ldots,m_k$ are sufficiently large numbers. Note that the divisor $L\sim_{\QQ}-\alpha K_{X_n}$. Let $\hat{L}$ be the proper transform of the divisor $L$. Consider $H=\hat{L}+\sum\epsilon_i E_i$, where $E_i$ are irreducible components of the exceptional divisor of $h$ and $\epsilon_i$ are positive numbers. Note that for sufficiently large $m_i$ and for sufficiently small $\epsilon_i$, the divisor $H$ is ample. Moreover, $X\setminus\Supp(H)\cong (Z\setminus\{q_1,\ldots,q_k\})\times\AAA^1$, where $q_1,\ldots,q_k$ are $k$ points on $Z$. So, $X$ has a $H$-polar cylinder.
 \end{proof}

 Let $X$ be a del Pezzo surface with du Val singularities. Assume that $\rho(X)=1$. Then $X$ has a $H$-polar cylinder if and only if $X$ has a $-K_X$-polar cylinder, where $H$ is an arbitrary ample divisor. On the other hand, there exist a classification of del Pezzo surfaces $X$ such that $X$ has a $-K_X$-polar cylinder (see {\cite{Ch1}}). By a classification of a del Pezzo surface $X$ has not cylinders if $X$ has one of the following collections of singularities: $4A_2, 2A_1+2A_3, 2D_4$. So, we may assume that $\rho(X)>1$ and $X$ has not cylinders. Run MMP for $X$. We obtain $$X=X_1\rightarrow X_2\rightarrow\cdots\rightarrow X_n.$$ By Lemma \ref{Lemm1} we may assume that $X_n$ is a del Pezzo surface with $\rho(X_n)=1$ and du Val singularities.
By Lemma \ref{Lemm2} we see that $X_n$ is a del Pezzo surface with one of the following collect of singularities: $4A_2, 2A_1+2A_3, 2D_4$. On the other hand, the surface $X$ has a smaller degree than $X_n$. But degree of $X_n$ is equal to one. So, $X=X_n$. On the other hand, $\rho(X_n)=1$, a contradiction.

This completes the proof of Theorem \ref{glav2}.


\begin{thebibliography}{lll}


\bibitem{Ch1}
Cheltsov I., Park J., Won J. \emph{Cylinders in singular del Pezzo surfaces} Comp. Math. \textbf{152} (2016), 1198 -- 1224.
\bibitem{Ch2}
Cheltsov I., Park J., Won J. \emph{Affine cones over smooth cubic surfaces} J. of the Eour. Math. \textbf{18} (2016), 1537 -- 1564.
\bibitem{Har}
Hartshorne R. \emph{Algebraic geometry} Springer Science+Business Media, Inc. (1977)
\bibitem{KeM}
Keel S., McKernan J. \emph{Rational curves on quasi-projective
surfaces}, Memoirs AMS \textbf{140} (1999), no. 669.
\bibitem{KPZ}
T. Kishimoto, Yu. Prokhorov, M. Zaidenberg, \emph{Group actions on affine cones}, Affine algebraic geometry, CRM Proc. Lecture Notes \textbf{54}, Amer. Math. Soc., 2011, pp. 123--163.
\bibitem{KPZ1}
T. Kishimoto, Yu. Prokhorov, M. Zaidenberg, \emph{$G_a$-actions on affine cones}, Transform. Groups, \textbf{18}, pp. 1137--1153.
\bibitem{Mat}
Matsuki K. \emph{Introduction to the Mori Program}, Springer-Verlag New York, Inc. (2002).



\end{thebibliography}
\end{document}